\documentclass[oneside,a4paper,11pt,notitlepage]{article}
\usepackage[left=0.9in, right=0.9in, top=1.5in, bottom=1.5in]{geometry}
\usepackage[T1]{fontenc} 
\usepackage[utf8]{inputenc} 
\usepackage[english]{babel} 
\usepackage{lipsum} 
\usepackage{lmodern}
\usepackage{amssymb}
\usepackage{amsthm}
\usepackage{bm}
\usepackage{mathtools}
\usepackage{xcolor}
\usepackage{dsfont}
\usepackage{enumerate}

\usepackage{yhmath}
\usepackage{stmaryrd}
\usepackage{amsmath}
\usepackage{tikz}
\usepackage{verbatim}
\usepackage{graphicx}
\usepackage{enumitem}
\usepackage{footnote} %

\usepackage{braket}
\usepackage{esint}
\newcommand{\abs}[1]{{\left|#1\right|}}
\newcommand{\norma}[1]{{\left\Vert#1\right\Vert}}

\usepackage{booktabs}
\usepackage{graphicx}
\usepackage{tikz}
\usetikzlibrary{patterns}
\usepackage{multicol}
\usepackage{caption}
\usepackage{enumerate}
\usepackage[skins,theorems]{tcolorbox}
\tcbset{highlight math style={enhanced,
		colframe=black,colback=white,arc=0pt,boxrule=1pt}}
\theoremstyle{definition}
\newtheorem{definizione}{Definition}[section]
\theoremstyle{plain}
\newtheorem{teorema}{Theorem}[section]

\newtheorem{lemma}[teorema]{Lemma}

\newtheorem{corollario}[teorema]{Corollary}
\theoremstyle{definition}
\newtheorem{esempio}{Example}[section]
\newtheorem{oss}[esempio]{Remark}
\newtheorem{open}[esempio]{Open problem}

\DeclareMathOperator{\R}{\mathbb{R}}

\DeclareMathOperator{\Om}{\Omega}

\usepackage[backend = bibtex, style = numeric, maxbibnames=99, maxcitenames = 99, 
maxbibnames = 99, maxalphanames = 99]{biblatex}
\AtEveryBibitem{%
    \clearfield{issn} 
    \clearfield{doi}
    \clearfield{eprint}
    \ifentrytype{online}{}{
    \clearfield{url}
  }
}
\usepackage{hyperref}
\hypersetup{linktoc=none, bookmarksnumbered, colorlinks=true, linkcolor=red}

\title{ A Rigidity Result for the Robin Torsion Problem}
\author{Alba Lia Masiello, Gloria Paoli*}
\date{}

\newcommand{\Addresses}{{
\bigskip 
  \footnotesize

  \textit{E-mail address}, A.L.~Masiello: \texttt{albalia.masiello@unina.it} 
   \medskip

 \textsc{Dipartimento di Matematica e Applicazioni ``R. Caccioppoli'', Universit\`a degli studi di Napoli Federico II, Via Cintia, Complesso Universitario Monte S. Angelo, 80126 Napoli, Italy.}
 
  \medskip 
  
   \textit{E-mail address}, G.~Paoli (Corresponding author)*: \texttt{gloria.paoli@fau.de} 
   
 \medskip
 
 \textsc{ Department of Data Science (DDS)
Chair in Dynamics, Control and Numerics (Alexander von Humboldt-Professorship),
Cauerstr. 11,
91058 Erlangen, Germany.}
 \medskip

 \par\nopagebreak 

}}

\begin{document}
\maketitle

\begin{abstract}
Let $\Omega \subset \mathbb{R}^2$ be an open, bounded and Lipschitz set. We consider the torsion problem for the Laplace operator associated to $\Omega$  with Robin boundary conditions. In this setting, we  study the equality case in the Talenti-type  comparison, proved in \cite{ANT}. We prove that the equality is achieved only if $\Omega$ is a disk and the torsion function $u$ is radial.
\end{abstract}

\Addresses

\section{Introduction}
Let $\beta>0$ and let  $\Om\subset\mathbb{R}^2$ be an open,  bounded and Lipschitz set. We consider the following problem for the Laplace operator:
\begin{equation}
    \label{rob1}
    \begin{cases}
    -\Delta u=1 & \text{in } \Om\\
    \displaystyle{\frac{\partial u}{\partial\nu}+\beta u=0} & \text{on } \partial\Omega,
    \end{cases}
\end{equation}
where $\nu$ is the outer unit normal to $\partial \Omega$.
A function $u\in H^1(\Omega)$ is a weak solution to \eqref{rob1} if 
\begin{equation}
\label{weak-f}
    \int _\Omega  \nabla u \nabla \varphi \, dx+ \beta \int_{\partial\Omega}u\varphi \, d \mathcal{H}^1 = \int_{\Omega} \varphi \, dx, \quad \forall \varphi \in H^1(\Omega).
\end{equation}
Classical arguments, see e.g  \cite{Dac},  ensure that there exists a positive and unique weak solution to \eqref{rob1}, that we denote by $u$. So, we can define the Robin torsional rigidity of $\Omega$ as the $L^1-$norm of $u$:
\begin{equation*}
    T(\Om):=\int_{\Om} u \, dx,
\end{equation*}
or, equivalently, as the maximum of the following Rayleigh quotient:
\begin{equation*}
    T(\Omega) = \max_{\substack{\varphi\in H^1(\Omega)\\ \varphi\not \equiv0} }  \frac{\displaystyle{\left(\int_\Omega \abs{\varphi(x)} \, dx\right)^{2}}}{\displaystyle{\int_\Omega \abs{\nabla \varphi(x)}^2 \, dx}+ \beta\int_{\partial\Omega} \varphi^2 \, d\mathcal{H}^1}.
\end{equation*}
In \cite{BG} the authors prove that the Robin torsional rigidity  is maximum on balls among bounded and Lipschitz  sets of fixed Lebesgue measure and the proof of this Saint-Venant type inequality relays on reflection arguments (see  also \cite{bugia}).

In the recent paper \cite{ANT}, the authors obtain the same result using symmetrization techniques.
They establish a Talenti-type comparison result between suitable Lorentz norms of the solution to the following problems:
\begin{equation*}
    \begin{cases}
    -\Delta u=f \, &\text{in } \Om, \\
    \displaystyle{\frac{\partial u}{\partial\nu}+\beta u=0} & \text{on } \partial\Omega,
    \end{cases} \quad \qquad\quad
    \begin{cases}
    -\Delta v=f^\sharp \, & \text{in } \Om^\sharp, \\
    \displaystyle{\frac{\partial v}{\partial\nu}+\beta v=0} & \text{on } \partial\Omega^\sharp,
    \end{cases}
\end{equation*}
where $f\in L^2(\Omega)$,  $f^\sharp$ is the \emph{Schwartz rearrangement} of $f$ (see Definition \ref{rear}) and $\Omega^\sharp$ is the ball centered at the origin having the same measure as $\Om$. Moreover, in the case $f\equiv 1$, they  obtain the following comparison result  in any dimension
\begin{equation}\label{comp}
    \norma{u}_{L^p(\Om)}\le \norma{v}_{L^p(\Om^\sharp)}, \, \,\quad p=1,2. 
\end{equation}
 We observe that, for $p=1$, inequality  \eqref{comp} is exactly the Saint-Venant inequality proved in  \cite{BG}. 
It is still an open problem  to establish if, for $p\in(1,+\infty)$, the ball maximizes the $L^p$ norm of the torsion function among open, bounded and  Lipschitz sets (see  \cite[Open Problem $1$]{bugia}). A first evidence in this direction is provided in \cite{R}, where it is proved that the ball is a critical shape for every $L^p$ norm in dimension $n>2$.

On the other hand, in the case $n=2$, the Open Problem $1$ contained in  \cite{bugia} is solved in \cite{ANT} in the following  stronger version:  
\begin{equation}
\label{punt}
    u^\sharp(x)\le v(x) \quad \forall x \in \Omega^\sharp, 
\end{equation}
where $u^\sharp$ is the Schwartz rearrangement of the solution to \eqref{rob1} and $v$ is the solution to
\begin{equation}
    \label{rob2}
    \begin{cases}
    -\Delta v=1 & \text{in } \Om^\sharp\\
    \displaystyle{\frac{\partial v}{\partial\nu}+\beta v=0} & \text{on } \partial\Om^\sharp.
    \end{cases}
\end{equation}
This kind of results in the Robin boundary setting was generalized to nonlinear case in \cite{AGM}, to  anisotropic case in \cite{San2}, with  mixed boundary conditions in  \cite{ACNT}, in the case of the Hermite operator in \cite{nunzia2022sharp} and for Riemannian manifolds in \cite{cinesi}. 

The aim of the present  paper is to characterize the equality case in \eqref{punt}, indeed we prove that the Talenti-type comparison is rigid in the planar case.

\begin{teorema}\label{Teo1}
Let $\Om\subset\R^2$ be an open, bounded  and  Lipschitz set and let $\Om^\sharp$ be the ball centered at the origin and  having the same measure as $\Om$. Let $u$ be the solution to \eqref{rob1} and let $v$ be the solution to \eqref{rob2}. If  $u^\sharp(x)=v(x)$ for  all $x\in \Om^\sharp$, then $$\Om=\Om^\sharp+ x_0, \quad u(\cdot+x_0)=u^\sharp(\cdot).$$
\end{teorema}
Moreover, we have the following extension of Theorem \ref{Teo1}.
\begin{teorema}\label{teo2}
Let $\Om\subset\R^2$ be an open, bounded and Lipschitz set and let $\Om^\sharp$ be the ball centered at the origin and having the same measure as $\Om$. Let $u$ be the solution to \eqref{rob1} and let $v$ be the solution to \eqref{rob2}. We denote by  $R$ the radius of $\Omega^\sharp$.\\
If  $\displaystyle{\min_\Omega u=\min_{\Omega^\sharp} v}$ and if there exists $r\in ]0, R[$ such that $u^\sharp(x)=v(x)$ for $\abs{x}=r$, then $$\Om=\Om^\sharp+ x_0, \quad u(\cdot+x_0)=u^\sharp(\cdot) \, \,  \text{in \;} \Omega.$$
\end{teorema}
The idea of the proof is the following. 
Starting from the proof in \cite{ANT} of  the pointwise comparison \eqref{punt}, we show that the equality  $u^\sharp=v$ implies that the level sets of $u$ are balls on the boundary of which the normal derivative of $u$ is constant. Then, we prove that these balls are concentric, using  an  argument inspired by \cite{Brothers1988} (see also \cite[Lemma 6]{serra}). In the Robin case,  the main difficulty is that, contrary to the Dirichlet case, the level sets of the solution may touch the boundary of $\Omega$.

As far as the  Dirichlet boundary conditions,  the starting point for the study of these kinds of problems is the paper by Talenti \cite{T}, in which a pointwise comparison is stated between the solution to the following problems:
\begin{equation*}
    \begin{cases}
    -\Delta u_D=f \, &\text{in } \Om, \\
     u_D=0 & \text{on } \partial\Omega,
    \end{cases} \quad \qquad
    \begin{cases}
    -\Delta v_D=f^\sharp \, & \text{in } \Om^\sharp, \\
    v_D=0 & \text{on } \partial\Omega^\sharp,
    \end{cases}
\end{equation*}
 whenever
 $f\in  L^{\frac{2n}{n+2}}(\Omega)$. 
 In particular, he proves in \cite{T} the pointwise inequality:
 \begin{equation}\label{point}
       u_D^\sharp(x)\le v_D(x) \quad \forall x \in \Omega^\sharp
 \end{equation}
 and, consequently, by integration,  the Saint-Venant inequality in the Dirichlet case holds:
\begin{equation*}
    \int_{\Omega} u_D \;dx=\int_{\Omega^\sharp} u^\sharp_D\;dx\leq \int_{\Omega^\sharp} v_D\;dx,
\end{equation*}
  conjectured by Saint-Venant in 1856.
 Moreover, a previous result in this direction is due to Weinberger, that proved in  \cite{W} the following result:
 $$\max_\Omega u_D\le \max_{\Om^\sharp} v_D.$$

We stress that, in the case of Dirichlet boundary conditions, the rigidity result holds and it is proved in \cite{lions_remark} (see Remark \ref{diff} for the main differences to the Robin case).  

Finally, we conclude by a list of generalization of Talenti's comparison results  in  different setting with Dirichlet boundary conditions.  Extension to the semilinear and nonlinear
elliptic case can be found, for instance, in \cite{T2}, to the  anisotropic elliptic operators in \cite{AFLT},  to  the  parabolic case in \cite{ALT2} and to  
higher order operators in \cite{AB,T3}.
We also refer the reader to \cite{Kaw, kes} and the references therein for  a survey on Talenti's techniques.

 The paper is organized as follows. In Section \ref{sec2} we recall some basic notions about rearrangements of functions and we recall some properties of the Torsion function, while  Section \ref{sec3} is dedicated to the proof of Theorem \ref{Teo1} and Theorem \ref{teo2} and to a list of open problems. 
 
\section{Notation and preliminaries}\label{sec2}
 Throughout this article, $|\cdot|$ will denote the Euclidean norm in $\mathbb{R}^2$,
 while $\cdot$ is the standard Euclidean scalar product. By $\mathcal{H}^1(\cdot)$, we denote the $1-$dimensional Hausdorff measure in $\mathbb{R}^2$. 
 The perimeter of $\Omega$ will be denoted by $P(\Omega)$ and since $\Omega$ is a bounded, open and Lipschitz set, we have that $P(\Omega)=\mathcal{H}^{1}(\partial\Omega)$. Moreover, we denote by $|\Omega|$ the Lebesgue measure of $\Omega$. 
 
 If $\Omega$ is an open and Lipschitz set, it holds the following coarea formula. Some references for results relative to the sets of finite perimeter and the coarea formula are, for instance, \cite{maggi2012sets,ambrosio2000functions}.

 \begin{teorema}[Coarea formula]
 Let $f:\Omega\to\R$ be a Lipschitz function and let $u:\Omega\to\R$ be a measurable function. Then,
 \begin{equation}
   \label{coarea}
   {\displaystyle \int _{\Omega}u|\nabla f(x)|dx=\int _{\mathbb {R} }dt\int_{(\Omega\cap f^{-1}(t))}u(y)\, d\mathcal {H}^{1}(y)}.
 \end{equation}
 \end{teorema}

We recall now some basic definitions and results about rearrangements and we refer to \cite{kes} for a general overview. 
 \begin{definizione}
	Let $u: \Omega \to \R$ be a measurable function, the \emph{distribution function} of $u$ is the function $\mu : [0,+\infty[\, \to [0, +\infty[$ defined by
	$$
	\mu(t)= \abs{\Set{x \in \Omega \, :\,  \abs{u(x)} > t}}.
	$$
\end{definizione}
\begin{definizione} 
	Let $u: \Omega \to \R$ be a measurable function, the \emph{decreasing rearrangement} of $u$, denoted by $u^\ast$, is the distribution function of $\mu $. 
	\end{definizione}
	\begin{oss}\label{inverse}
	We observe that the function $\mu(\cdot)$ is decreasing and right continuous and the function $u^\ast(\cdot)$ is the generalized inverse of the function $\mu(\cdot)$.
	\end{oss}

	\begin{definizione}\label{rear}
	 The \emph{Schwartz rearrangement} of $u$ is the function $u^\sharp $ whose level sets are balls with the same measure as the level sets of $u$. 
	\end{definizione}
		We have the following relation between $u^\sharp$ and $u^*$:
	$$u^\sharp (x)= u^*(\pi\abs{x}^2)$$
 and it can be easily checked that the functions $u$, $u^*$ e $u^\sharp$ are equi-distributed, so we have that
$$ \displaystyle{\norma{u}_{L^p(\Omega)}=\norma{u^*}_{L^p(0, \abs{\Omega})}=\lVert{u^\sharp}\rVert_{L^p(\Omega^\sharp)}}.$$
Let now $u$ be the solution to \eqref{rob1}. For $t\geq 0$, we introduce the following notations:
$$U_t=\left\lbrace x\in \Omega : u(x)>t\right\rbrace \quad \partial U_t^{int}=\partial U_t \cap \Omega, \quad \partial U_t^{ext}=\partial U_t \cap \partial\Omega, \quad \mu(t)=\abs{U_t}$$
and, if $v$ is the solution to \eqref{rob2}, using the same notations as above,  we set
$$V_t=\left\lbrace x\in \Omega^\sharp : v(x)> t\right\rbrace, \quad \partial V_t^{int}=\partial V_t \cap \Omega, \quad \partial V_t^{ext}=\partial V_t \cap \partial\Omega, \quad \phi(t)=\abs{V_t}.$$
Because of the invariance of the Laplacian under rotation, we have that $v$ is radial.
Moreover, we observe that the solutions $u$ to \eqref{rob1} and $v$ to \eqref{rob2} are both superharmonic and so, by the strong maximum principle, it follows that they achieve their minima on the boundary.

From now on, we denote by 
\begin{equation}
    u_m = \min_\Omega u, \quad \quad v_m=\min_{\Omega^\sharp} v,
\end{equation}
\begin{equation}
    u_M = \max_\Omega u, \quad \quad v_M=\max_{\Omega^\sharp} v.
\end{equation}
Since we are assuming that the Robin boundary parameter $\beta$ is strictly positive, 
 we have that $u_m > 0$ and $v_m > 0$. Hence, $u$ and $v$ are strictly positive in the interior of $\Omega$.
 
 Since $v$ is radial, positive and decreasing along the radius then, for $0\le t\le v_m$, $$V_t=\Omega^\sharp,$$ while, for $v_m<t<v_M$, we have that $V_t$ is a ball concentric to $\Omega^\sharp$ and strictly contained in it.
 
 In the next remarks, we collect some general and useful results. 
 \begin{oss}
 By the weak formulation \eqref{weak-f} and the isoperimetric inequality, we have  that 
	\begin{equation*}
		\begin{split}
			 v_m  \text{P}(\Omega^\sharp) &= \int_{\partial \Omega^\sharp} v(x) \, d\mathcal{H}^1= \frac{1}{\beta}\int_{\Omega^\sharp}  \, dx=\frac{1}{\beta} \int_{\Omega} \, dx \\
			& = \int_{\partial \Omega} u(x) \, d\mathcal{H}^1 
		\geq u_m  \text{P}(\Omega)  \geq  u_m \text{P}(\Omega^\sharp),
		\end{split}
	\end{equation*}
	and, as a consequence,
	\begin{equation}
		\label{minima_eq}
		u_m \leq  v_m.
	\end{equation}
Moreover, from \eqref{minima_eq}  follows that 
	\begin{equation}
		\label{mf}
		\mu (t) \leq \phi (t) = \abs{\Omega} \quad \forall t \leq v_m.
	\end{equation} 
	 \end{oss}

\begin{oss}\label{ziemer}
We observe that $\phi$, the distribution function of $v$, is absolutely continuous. Indeed, in \cite[Lemma 2.3]{Brothers1988}, is proved that the absolutely continuity of $\phi$ is equivalent to the following condition: 
\begin{equation}\label{brother}
    \abs{\Set{\nabla v=0}\cap v^{-1}(v_m,v_M)}=0 \end{equation}
which is verified by $v$, as its gradient never vanishes on the level sets $V_t$.
\end{oss}

The starting point of the proof of our main results is the following Lemma, proved in \cite{ANT}. For the convenience of exposition, we report here the proof.
\begin{lemma} \label{key}
Let $u$ be a solution to \eqref{rob1} and let $v$ be a solution to \eqref{rob2}. Then, for almost every $t >0$, we have
	\begin{equation} 
	\label{talentimu}
	4\pi  \leq  \left( - \mu'(t) + \frac{1}{\beta }\int_{\partial U_t^{ext}} \frac{1}{u} \, d\mathcal{H}^1\right)
	\end{equation}
	and
\begin{equation} 
	\label{talentiphi}
	4\pi  =  \left( - \phi'(t) + \frac{1}{\beta }\int_{\partial V_t^{ext}} \frac{1}{v} \, d\mathcal{H}^1\right).
	\end{equation}
		
\end{lemma}
\begin{proof}
 Let $t >0$ and  $h >0$. Let us choose the following test function in the weak  formulation \eqref{weak-f}  
	\begin{equation*}
	\left.
	\varphi (x)= 
	\right.
	\begin{cases}
	0 & \text{ if }u < t \\
	u-t & \text{ if }t< u < t+h \\
	h & \text{ if }u > t+h.
	\end{cases} 
	\end{equation*}
	Then, we have
	\begin{equation}\label{weaky}
	\begin{split}
	\int_{U_t \setminus U_{t+h}} \abs{\nabla u}^2\, dx &+ \beta h \int_{\partial U_{t+h}^{ ext }} u  \, d\mathcal{H}^1 + \beta \int_{\partial U_{t}^{ ext }\setminus \partial U_{t+h}^{ ext }} u (u-t) \, d\mathcal{H}^1 \\
	& =  \int_{U_t \setminus U_{t+h}}  (u-t ) \, dx + h \int_{U_{t+h}}  \, dx.
	\end{split}
	\end{equation}
	Dividing \eqref{weaky} by $h$, using coarea formula \eqref{coarea} and letting $h$ go to $0$, we have that for a.e. $t>0$
	\begin{equation*}
	\int_{\partial U_t} g(x) \, d\mathcal{H}^1 = \int_{U_{t}} \, dx ,
	\end{equation*}
	where
	\begin{equation}\label{fung}
	    \left.
	g(x)= 
	\right.
	\begin{cases}
	\abs{\nabla u } & \text{ if }x \in \partial U_t^{ int },\\
	\beta u & \text{ if }x \in \partial U_t^{ ext }.
	\end{cases}
	\end{equation}
Using the isoperimetric inequality, for a.e. $t\in [0, u_M )$ we have 

	\begin{align}
	\label{isop}
	    2\sqrt{\pi}  \mu(t)^{\frac{1}{2}} &\leq P(U_t) = \int_{\partial U_t} \,  d\mathcal{H}^1\leq\\
	    \label{holder}
	    &\leq \left(\int_{\partial U_t}g\, d\mathcal{H}^1\right)^{\frac{1}{2}} \left(\int_{\partial U_t}\frac{1}{g} \, d\mathcal{H}^1\right)^{\frac{1}{2}} \\
	&= \mu(t)^{\frac{1}{2}} \left( \int_{\partial U_t^{ int }}\frac{1}{\abs{\nabla u}}\, d\mathcal{H}^1 +\frac{1}{\beta} \int_{\partial U_t^{ ext }}\frac{1}{u} \,  d\mathcal{H}^1 \right)^{\frac{1}{2}}.
	\end{align}
	and, so,  \eqref{talentimu} follows. Finally, we notice that, if $v$ is the solution to \eqref{rob2}, then all the inequalities above are equalities, and, consequently, we have \eqref{talentiphi}.
\end{proof}

\begin{oss}
By integrating \eqref{talentiphi}, it is possible to write the explicit expression of $v$, that is 
$$v(x)= \frac{\abs{\Om}-\pi\abs{x}^2}{4\pi}+ \frac{\abs{\Om}^{\frac{1}{2}}}{2\sqrt{\pi}\beta}.$$
\end{oss}

\begin{oss} Integrating \eqref{talentimu} and \eqref{talentiphi} between $0$ and $t$ and integrating by parts, 
it is proved in  \cite{ANT}  that
\begin{equation}\label{mu_phi}
  \mu(t)\leq \phi(t), \quad t\geq v_m.
\end{equation}
Finally, we observe that 
the pointwise comparison \eqref{punt} easily follows from \eqref{mu_phi}.
\end{oss}

\section{Proof of the main results}\label{sec3}
\begin{proof}[Proof of Theorem \ref{Teo1}]
First of all, let us observe that, from the fact that we are assuming that $u^\sharp=v$, we have 
\begin{equation}\label{min}
     u_m=v_m.
\end{equation}
We integrate now \eqref{talentimu} and \eqref{talentiphi} from $0$ to $t$ and, since $u^\ast$ is the generalized inverse of $\mu$ (Remark \ref{inverse}), we perform the following change of variables $\mu(t)=s$ and $\phi(t)=s$. So, we get
\begin{equation}
    \label{vst}
    v^\ast(s)=\frac{\abs{\Om}-s}{4\pi}+ \frac{\abs{\Om}^{\frac{1}{2}}}{2\sqrt{\pi}\beta}
\end{equation}
\begin{equation}
    \label{ust}
    u^\ast(s)\le \frac{\abs{\Om}-s}{4\pi}+ \frac{1}{4\pi \beta}\int_0^{u^\ast(s)}dr \int_{\partial U_r^{ext}}\frac{1}{u}\, d\mathcal{H}^1.
\end{equation}
From $u^\sharp=v$, we have $u^\ast=v^\ast$ and, so, combining \eqref{vst} and \eqref{ust}, we get
\begin{equation}\label{chain}
    \begin{aligned}
    \frac{\abs{\Om}^{\frac{1}{2}}}{2\sqrt{\pi}\beta}&\le \frac{1}{4\pi \beta}\int_0^{u^\ast(s)}dr \int_{\partial U_r^{ext}}\frac{1}{u}\, d\mathcal{H}^1\\
    &\le \frac{1}{4\pi \beta u_m} \int_{0}^{u_M}\int_{\partial U_r^{ext}} \, d\mathcal{H}^1=  \frac{1}{4\pi \beta u_m} \frac{\abs{\Om}}{\beta}=\frac{\abs{\Om}^{\frac{1}{2}}}{2\sqrt{\pi}\beta},
\end{aligned}
\end{equation}

where the last equality follows from \eqref{min}.
 Therefore, all the inequalities in \eqref{chain} are equalities and, consequently, equality holds in \eqref{talentimu}.

We now divide the proof in two steps. 

\textbf{Step 1}. Let us prove that every level set $\Set{u>t}$ is a ball.

Equality in \eqref{talentimu} implies the equality in \eqref{isop},  i.e.
$$  2\sqrt{\pi}  \mu(t)^{\frac{1}{2}} = P(U_t)$$
that means that almost every level set is a ball. On the other hand, for all $t\in [u_m, u_M)$, there exists a sequence $\Set{t_k}$ such that
\begin{enumerate}
    \item $t_k\to t$;
    \item $t_k>t_{k+1}$;
    \item $\{u>t_k\}$ is a ball for all $k$. 
\end{enumerate}
Since $\Set{u>t}=\cup_k\Set{u>t_k}$ can be written as an increasing union of balls, then we have that $\{u>t\}$ is a ball for all $t$ and, from the fact that  $\Om=\{u>u_m\}$, we obtain that $\Om=x_0+\Om^\sharp$. From now on, we can assume without loss of generality that $x_0=0$.

\textbf{Step 2.} Let us prove that the level sets are concentric balls. 

Equality in \eqref{talentimu} implies also equality in \eqref{holder}, i.e. 
$$\int_{\partial U_t} \,  d\mathcal{H}^1=\left(\int_{\partial U_t}g\, d\mathcal{H}^1\right)^{\frac{1}{2}} \left(\int_{\partial U_t}\frac{1}{g} \, d\mathcal{H}^1\right)^{\frac{1}{2}}.$$
 This means that, as we have equality in the H\"older inequality,  for almost every $t$, the function
$$
	\left.
	g(x)= 
	\right.
	\begin{cases}
	\abs{\nabla u } & \text{ if }x \in \partial U_t^{ int },\\
	\beta u & \text{ if }x \in \partial U_t^{ ext }
	.
	\end{cases}
	$$
is constant, in particular	
\begin{equation}
   \label{conf1} 
\abs{\nabla u}= C_t, \quad\forall x\in \partial U_t^{int}, \qquad \beta u=C_t, \quad \forall x\in \partial U_t^{ext},
\end{equation}
and by continuity we can infer that this is true for all $t$.
By the way, we observe that for all $x\in \partial U_t$,
\begin{equation}\label{normal}
    g(x)=\displaystyle{\frac{\partial u(x)}{\partial \nu_t}},
\end{equation}
where $\nu_t$ is the unit outer normal to $\partial U_t$.

     From equality \eqref{talentimu}, we have also that
   $$\mu(t)=\phi(t),$$ and, consequently, we can deduce from Remark \ref{ziemer} that also $\mu$ is absolutely continuous. If we denote by 
    $$B(x(t),\rho(t))=\Set{u>t},$$
    we can observe that the function $\mu(t)$ is locally Lipschitz in $(u_m, u_M)$, and, so, the function $$\rho(t)=\left(\frac{\mu(t)}{\pi}\right)^{\frac{1}{2}}$$ is also locally Lipschitz. Moreover, since $\Set{u>t}\subseteq \Set{u>s}$ for $t>s$, we have 
    $$\abs{x(t)-x(s)}\le \rho(s)-\rho(t)$$
    and, consequently, $x(t)$ is locally Lipschitz.
    
  Let us assume now by contradiction that $x(t)$ is not constant. This means that there exists $t_0\in(u_m,u_M)$ such that
    $$y=\frac{d}{dt}x(t_0)\neq 0.$$
    Let us set $z:=y/\abs{y}$ and
    $$P(t):=x(t)+\rho(t)z\in \partial B(x(t),\rho(t)), \quad Q(t):=x(t)-\rho(t)z\in\partial B(x(t),\rho(t)).$$
  We have that, for all $t\in(u_m,u_M)$, 
    \begin{equation}\label{const}
        u(P(t))=u(Q(t))=t
    \end{equation}
     and
    $$\frac{\partial u (P(t_0))}{\partial\nu_{t_0}}=\nabla u(P(t_0))\cdot z$$
    $$-\frac{\partial u (Q(t_0))}{\partial\nu_{t_0}}=\nabla u(Q(t_0))\cdot z.$$
  On the other hand, from \eqref{const}, we obtain 
    \begin{equation*}
        1=\frac{d}{dt}u(P(t))\lvert_{t_0}= \nabla u(P(t_0))\cdot P'(t_0)=\nabla u(P(t_0))\cdot z (\abs{y}+\rho'(t_0))
    \end{equation*}
    \begin{equation*}
        1=\frac{d}{dt}u(Q(t))\lvert_{t_0}= \nabla u(Q(t_0))\cdot Q'(t_0)=\nabla u(Q(t_0))\cdot z (\abs{y}-\rho'(t_0)),
    \end{equation*}
  and, consequently,
    \begin{equation}\label{conf2}
        \frac{\partial u}{\partial \nu_{t_0}}(P(t_0))(\abs{y}+\rho'(t_0)=-\frac{\partial u}{\partial \nu_{t_0}}(Q(t_0))(\abs{y}-\rho'(t_0).
    \end{equation}
    Moreover, by \eqref{conf1} we have 
    $$\frac{\partial u}{\partial \nu_{t_0}}(P(t_0))=\frac{\partial u}{\partial \nu_{t_0}}(Q(t_0))$$
    and, so, we have $|y|=0$, that is absurd.   

Thus, we have proved that $u$ is radially symmetric and, since 
$$\frac{\partial u}{\partial r}=\frac{\partial u}{\partial \nu}<0,$$
$u$ is decreasing along the radii and $u=u^\sharp$.
\end{proof}

\begin{oss}\label{diff}
In the proof of Theorem 1.1 the main difference 
 from the proof of the rigidity result in the Dirichlet case contained in  \cite{lions_remark} is Step 2. Indeed, in \cite{lions_remark}, the authors use the steepest descent lines method, which relays on the fact that $\abs{\nabla u}$ is constant on the level set of $u$, which is not a priori true in the Robin case. 
\end{oss}

\begin{proof}[Proof of Theorem \ref{teo2}]
Let us set $s=\pi r^2$. The assumption $u^\sharp(x)=v(x)$ for $\abs{x}=r$, implies
\begin{equation*}
   u^\ast(s)=v^\ast(s).
\end{equation*}
Arguing now as in the proof of Theorem \ref{Teo1}, we have
\begin{equation*}
    \begin{aligned}
    \frac{\abs{\Omega}-s}{4\pi}+\frac{\abs{\Om}^{\frac{1}{2}}}{2\sqrt{\pi}\beta}=v^\ast(s)=u^\ast(s)\le \frac{\abs{\Omega}-s}{4\pi} +\int_0^{u^\ast(s)}dr\int_{\partial U_r^{ext}} \frac{1}{u}\, d\mathcal{H}^1\le \\\leq 
    \frac{\abs{\Omega}-s}{4\pi} + \frac{1}{4\pi \beta u_m} \frac{\abs{\Om}}{\beta}= \frac{\abs{\Omega}-s}{4\pi}+\frac{\abs{\Om}^{\frac{1}{2}}}{2\sqrt{\pi}\beta},
    \end{aligned}
\end{equation*}
where in the last equality we have used the hypothesis $u_m=v_m$.
So, we have equality in \eqref{ust} and, consequently,  in \eqref{talentimu} for $\overline{t}:=u^\ast(s)$. 
As before, this implies that
\begin{itemize}
    \item $\{u> \overline{t}\}$ is a ball;
    \item $\mu(\overline{t})=\phi(\overline{t})$;
    \item the function $g$ defined in \eqref{fung} is constant on $\partial U_{\overline{t}}$.
\end{itemize}
Let us observe that, for all $\tau>v_m$
\begin{equation}
    \label{muint}
    \begin{aligned}
    &\int_0^\tau t \left(\int_{\partial U_t^{\text{ext}}} \frac{1}{u(x)} \, d\mathcal{H}^1\right)\, dt\le \int_0^{u_M} t \left(\int_{\partial U_t^{\text{ext}}} \frac{1}{u(x)} \, d\mathcal{H}^1\right) \, dt=\\
    &\int_{\partial\Omega} \left(\int_0^{u(x)} \frac{t}{u(x)}\, dt\right)\, d\mathcal{H}^1= \int_{\partial\Omega} \frac{u(x)}{2}= \frac{\abs{\Omega}}{2\beta},
    \end{aligned}
\end{equation}
while, for $v$ it holds

\begin{equation}
    \label{fiint}
    \int_0^\tau t \left(\int_{\partial V_t^{\text{ext}}} \frac{1}{v(x)} \, d\mathcal{H}^1\right)\, dt = \int_0^{v_m} t \left(\int_{\partial V_t^{\text{ext}}} \frac{1}{v(x)} \, d\mathcal{H}^1\right) \, dt=\frac{v_m P(\Omega^\sharp)}{2}=\frac{\abs{\Omega}}{2\beta},
\end{equation}
where the first equality follows from the fact that $\forall t> v_m$
$$\partial V_t^{\text{ext}}=\partial V_t\cap \partial\Omega=\emptyset.$$
If we multiply \eqref{talentimu} and \eqref{talentiphi} by $t$ and we integrate from $0$ to $\overline{t}$, we get
\begin{equation}\label{1}
    2\pi \overline{t}^2 \leq \int_0^{\overline{t}} t \left(- \mu'(t) + \frac{1}{\beta }\int_{\partial U_t^{ext}} \frac{1}{u(x)} \, d\mathcal{H}^1\right) \, dt \le \int_0^{\overline{t}} t \left(- \mu'(t)\right)\, dt + \frac{\abs{\Omega}}{2 \beta^2},
\end{equation}
where in the last inequality we use \eqref{muint}, and  we get
\begin{equation}\label{2}
    2\pi \overline{t}^2 = \int_0^{\overline{t}} t \left(- \phi'(t) + \frac{1}{\beta }\int_{\partial V_t^{ext}} \frac{1}{v(x)} \, d\mathcal{H}^1\right) \, dt = \int_0^{\overline{t}} t \left(- \phi'(t)\right)\, dt + \frac{\abs{\Omega}}{2 \beta^2},
\end{equation}
where in the last equality we use  \eqref{fiint}.
Therefore, combining \eqref{1} and \eqref{2}, we have that 
\begin{equation}\label{3}
    \int_0^{\overline{t}} t \left(- \mu'(t)\right)\, dt\ge \int_0^{\overline{t}} t \left(- \phi'(t)\right)\, dt,
\end{equation}
and, integrating by parts and recalling that $\mu(\overline{t})=\phi(\overline{t})$, we get
\begin{equation*}
    \int_0^{\overline{t}} \left( \mu(t)-\phi(t)\right)dt\ge 0. 
\end{equation*}
On the other hand, since \eqref{mu_phi} holds for all $t\ge 0$,  we have
\begin{equation*}
    \mu(t) =\phi (t),  \quad \forall t\in [0, \overline{t}]
\end{equation*}
 and this implies that equality holds in \eqref{talentimu} for all $t\in [0, \overline{t}]$.
Now, arguing as in Theorem \ref{Teo1}, 
we recover $\Omega=\Omega^\sharp+ x_0$ and $u(\cdot+x_0)=u^\sharp(\cdot)$ in $\Set{ r\le \abs{x}\le R}$. Finally, for the uniqueness of the solution to problem \eqref{rob2}, once we have that $\Omega$ is a ball, it follows that $u=v$ for all $x\in \Omega$. 
\end{proof}

As a particular case of the above result, if we take $r=0$, we have

\begin{corollario}
Let $\Om\subset\R^2$ be an open, bounded and Lipschitz set and let $\Om^\sharp$ be the ball, centered at the origin, having the same measure of $\Om$. Let $u$ be the solution to \eqref{rob1} and let $v$ be the solution to \eqref{rob2}. 
If $u_m=v_m$, and $u_M=v_M$, then $$\Om=\Om^\sharp+ x_0, \quad u(\cdot+x_0)=u^\sharp(\cdot) \, \,  \text{in } \Omega^\sharp.$$
\end{corollario}

\begin{open}
\begin{itemize} Below we present a list of open problems and work in progress.
    \item Generalize the results contained in Theorem \ref{Teo1} and Theorem \ref{teo2} to higher dimension. In order to do that, one should prove \eqref{punt} in $\R^n$ for $n\ge 3$ (we adress to Open Problem $1$ in \cite{ANT}).
     \item Generalize the results contained in Theorem \ref{Teo1} and Theorem \ref{teo2} under weaker assumptions.
    \item Generalize the previous results to the $p-$Torsion or to the anisotropic Torsion.
\end{itemize}
\end{open}

\section*{Statements and Declarations}

\noindent \textbf{Funding:} The author  Gloria Paoli is supported by the Alexander von Humboldt Foundation with an Alexander von Humboldt research fellowship. 
\noindent The authors Alba Lia Masiello and Gloria Paoli are supported by GNAMPA of INdAM. 

\noindent \textbf{Conflict of interst:} There is no conflict of interest to disclose

\addcontentsline{toc}{chapter}{Bibliografia}

\printbibliography[heading=bibintoc, title={References}]

\end{document}